\documentclass[12pt,a4paper]{amsart}
\usepackage{amsfonts}
\usepackage{amsthm}
\usepackage{amsmath}
\usepackage{amsrefs}
\usepackage[toc,page]{appendix}
\usepackage[foot]{amsaddr}
\usepackage{amscd}
\usepackage[latin2]{inputenc}
\usepackage{t1enc}
\usepackage{amssymb}
\usepackage[mathscr]{eucal}
\usepackage{indentfirst}
\usepackage{graphicx}
\usepackage{graphics}
\usepackage{pict2e}
\usepackage{epic}
\numberwithin{equation}{section}
\usepackage[margin=2.9cm]{geometry}
\usepackage{epstopdf} 
\usepackage[all,cmtip]{xy}
\usepackage{enumerate}
\usepackage{pdfpages}
\usepackage[colorlinks = true,
            linkcolor = blue,
            urlcolor  = magenta,
            citecolor = blue,
            anchorcolor = blue]{hyperref}

\usepackage{relsize} 
\usepackage[bbgreekl]{mathbbol} 
\usepackage{mathtools}
\usepackage{amsfonts} 
\DeclareSymbolFontAlphabet{\mathbb}{AMSb} 
\DeclareSymbolFontAlphabet{\mathbbl}{bbold} 
\newcommand{\Prism}{{\mathlarger{\mathbbl{\Delta}}}}
 
\usepackage{tikz}
\usetikzlibrary{matrix,arrows,decorations.pathmorphing}

\theoremstyle{remark}    
\newtheorem*{sketch}{Sketch of the proof}  
\theoremstyle{plain}
\newtheorem{Th}{Theorem}[section]
\newtheorem{Lemma}[Th]{Lemma}

\newtheorem{Prop}[Th]{Proposition}

\theoremstyle{definition}
\newtheorem{Def}[Th]{Definition}

\newtheorem{Rem}[Th]{Remark}
\newtheorem{Ques}[Th]{Question}
\newtheorem{?}[Th]{Problem}
\newtheorem{Ex}[Th]{Example}

\newcommand{\perf}{{\rm{perf}}}
\newcommand{\Tor}{{\rm{Tor}}}
\newcommand{\Spec}{{\rm{Spec}}}
\newcommand{\Spf}{{\rm{Spf}}}
\newcommand{\Rad}{{\rm{Rad}}}
\newcommand{\colim}{{\rm{colim}}}

\newcommand{\stacktag}[1]{\cite[\href{https://stacks.math.columbia.edu/tag/#1}{Tag #1}]{stacks-project}}

\begin{document}

\setcounter{page}{1}

\title{Prismatic Cohomology and De~Rham--Witt Forms}

\author{Semen Molokov}
\address{HSE University, Department of Mathematics, Ulitsa Usacheva 6, Moscow 119048, Russia}
\email{sam\_molokov@yahoo.com}

\setcounter{tocdepth}{1}

\begin{abstract}
For any prism $(A, d)$, we construct an analogue of Fontaine's map $W_r(A/d) \to A/d\phi(d)\cdots\phi^{r-1}(d)$. Subsequently, we define a canonical map from de~Rham--Witt forms to prismatic cohomology in the perfect case and prove that it is an isomorphism. Using this result, we obtain an explicit description of the prismatic cohomology $H^i((S/A)_\Prism,\mathcal{O}_\Prism/d\phi(d)\cdots\phi^{n-1}(d))$, where $S$ is the $p$-completion of a polynomial algebra over $A/d$.
\end{abstract}

\maketitle



\tableofcontents
\section{Introduction}
This paper focuses on prismatic cohomology, a recently developed cohomology theory for schemes over $p$-adic rings, introduced by B.~Bhatt and P.~Scholze in~\cite{Prisms}. This theory specializes to most known integral $p$-adic cohomology theories. In the particular case when the base ring is $R = \mathcal{O}_C$, where $C$ is a complete algebraically closed extension of $\mathbb{Q}_p$, the similar construction was previously obtained by B.~Bhatt, M.~Morrow, and P.~Scholze in~\cite{BMS_I}, where they establish a comparison isomorphism with continuous de~Rham--Witt forms. However, this result has not yet been extended to the general setting of prismatic cohomology.

In this paper, we aim to make progress in this direction. We begin with the following question:

\begin{Ques}\label{main}
Let $(A,d)$ be an oriented prism, and let $S$ be a $p$-completely smooth $A/d$-algebra. Then, for any integer $r\geq 1$, there is a functorial isomorphism of $A/d\phi(d)\cdots\phi^{r-1}(d)$-modules
\[W_r\Omega^{i, cont}_{S/(A/d)}\widehat{\otimes}^L_{W_r(A/d)}A/d\phi(d)\cdots \phi^{r-1}(d) \to H^i((S/A)_\Prism,\mathcal{O}_\Prism/d\phi(d)\cdots\phi^{r-1}(d)).
\]
\end{Ques}

If this statement holds, it requires that, for any oriented prism $(A, d)$ and all integers $r\geq 1$, there are canonical maps
\[
W_r(A/d)\to A/d\phi(d)\cdots\phi^{r-1}(d).
\]
The construction of these maps is the first result of this paper. Our approach relies on reduction to the case of perfect prisms, where the existence of such maps was previously established in \cite{BMS_I}. We also present an alternative construction, provided by an anonymous referee, in Appendix~\ref{appendix}.

After constructing these maps, it is then a formal procedure to obtain a map
\[ W_r(S)\to H^0((S/A)_\Prism, \mathcal{O}_{\Prism}/d\phi(d)\cdots\phi^{r-1}(d)).
\]

Moreover, similar to the Hodge--Tate complex $\overline{\Prism}_{S/A}$, the prismatic cohomology groups $\bigoplus_{n\geq 0} H^n((S/A)_\Prism, \mathcal{O}_\Prism/d\phi(d)\cdots\phi^{r-1}(d))\{n\}$ naturally form a commutative differential graded algebra, with a Bockstein homomorphism serving as the differential. From this, one can conclude that the desired maps from de~Rham--Witt forms must be unique if they exist.

The next result of this paper is the existence of these maps for perfect prisms. In this case, we prove the following theorem:

\begin{Th}\label{perfect_case}
Let $(A,d)$ be a perfect prism such that $A/d$ is $p$-torsion-free, and let $S$ be a $p$-completely smooth $A/d$-algebra. Then we have the functorial isomorphism
\[W_n\Omega_{S/(A/d)}^{i, cont}\to H^i(\Prism_{S/A}\otimes^L_A A/d\phi(d)\cdots\phi^{n-1}(d)).
\]
\end{Th}

This theorem is proved by reduction to the result of B.~Bhatt, M.~Morrow and P.~Scholze in~\cite{BMS_I}, using Andr\'e's lemma (see \cite[Theorem~7.14]{Prisms}).

In the general case, Peter Scholze has informed us that the statement in Question~\ref{main} is false. It remains unclear how the left-hand side ought be modified. Nevertheless, when $S=A/d\langle T_1,\ldots, T_n\rangle$ is a $p$-completed polynomial algebra, both sides of Question~\ref{main} admit a similar explicit description. Indeed, it is known from \cite{LZ} that $W_r\Omega^i_{S/(A/d)}$ is a certain infinite direct sum of copies of $W_i(A/d)$ for $1\leq i\leq r$, and it turns out that $H^i((S/A)_\Prism, \mathcal{O}_\Prism/d\phi(d)\cdots\phi^{r-1}(d))$ is a $p$-completed infinite direct sum, with the same index set, of copies of $A/d\phi(d)\cdots\phi^{i-1}(d)$. This constitutes the third main result of the paper.

\subsection{Acknowledgments} We are deeply grateful to Peter Scholze for proposing this question, as well as for his invaluable help and useful comments. We also would like to thank Vadim Vologodsky for his interest in the project and the anonymous referee for the detailed report containing the alternative proof of Proposition~\ref{map}.

\section{Prisms and distinguished elements}
Fix a prime number $p$. In this section, we present fundamental definitions and results from \cite{Prisms} that we will use extensively throughout this paper. We also draw significantly from the lectures in \cite{Bhatt}.

\subsection{$\delta$-rings}
This subsection presents key facts about $\delta$-rings, a concept introduced by Joyal in \cite{Joy}. This notion provides an effective framework for studying rings with a lift of Frobenius modulo $p$. For a comprehensive treatment of this theory, we refer to \cite{Borg}.

\begin{Def}
A $\delta$-ring is a pair $(A,\delta)$ where $A$ is a commutative ring and $\delta \colon A\to A$ is a map of sets with $\delta(0)=\delta(1)=0$, satisfying the following two identities:
\[
  \delta(xy)= x^p\delta(y)+y^p\delta(x) + p\delta(x)\delta(y)
\]
and
\[
\delta(x+y) = \delta(x)+\delta(y)+\frac{x^p+y^p-(x+y)^p}{p}.
\]
\end{Def} 

In the literature, a $\delta$-structure is also known as a $p$-derivation. The main feature of the $\delta$-structure is that it provides a Frobenius lift:

\begin{Lemma}\label{FROB} 
Let $A$ be a commutative ring.

\begin{enumerate}[(i)]
  \item If $\delta\colon A\to A$ provides a $\delta$-structure on $A$, then the map $\phi\colon A\to A$ defined by $\phi(x)=x^p+p\delta(x)$ is an endomorphism of $A$ that lifts the Frobenius on $A/p$.

  \item  When $A$ is $p$-torsion-free, the construction described above gives a bijection between $\delta$-structures on $A$ and Frobenius lifts on $A$.
  
  \item If $A$ is a $\delta$-ring, then $\phi\colon A \to A$ is a $\delta$-map, i.e., $\phi(\delta(x))=\delta(\phi(x))$ for any $x\in A$.
\end{enumerate}
\end{Lemma}
The proof is standard and can be found in \cite{Bhatt} and \cite{Prisms}. 

\begin{Ex} Lemma~\ref{FROB} provides several simple examples of $\delta$-rings when $A$ is $p$-torsion-free:
\begin{enumerate}[(i)]
    \item The ring $\mathbb{Z}$ with identity $\phi(x)=x$. It is easy to see that this ring is an initial object in the category of $\delta$-rings.
    \item Let $A=\mathbb{Z}[x]$. For any $g\in A$, there exists a unique Frobenius lift $\phi$ determined by $\phi(x) = x^p + pg$. Thus, this yields a unique $\delta$-structure on $A$ with $\delta(x)=g$.
    \item If $k$ is a perfect field of characteristic $p>0$, then the ring of Witt vectors $W(k)$ admits a unique standard lift of Frobenius, which induces a unique $\delta$-structure on $W(k)$.
\end{enumerate}
\end{Ex}

\begin{Def}
A $\delta$-ring $A$ is called perfect if $\phi$ is an isomorphism.
\end{Def}

\subsection{Prisms and distinguished elements}
In this subsection, we discuss definitions and some properties of prisms and distinguished elements. We assume that all rings that appear are $p$-local, i.e., $p\in \Rad(A)$ where $\Rad(A)$ is the Jacobson radical of $A$. This condition is clearly satisfied when $A$ is $p$-adically complete.

\begin{Def}
Let $A$ be a $\delta$-ring. An element $d\in A$ is called distinguished or primitive if $\delta(d)$ is a unit in $A$.
\end{Def}

Note that since $\delta$ commutes with $\phi$ by Lemma~\ref{FROB} and all our rings are $p$-local, we see that $d$ is distinguished if and only if $\phi(d)$ is distinguished.

\begin{Ex} The main examples related to cohomology theories are:
\begin{enumerate}[(i)]
    \item  Let $(A, d)$ be $(\mathbb{Z}_p, p)$. Then $p$ is distinguished since $\delta(p)=1-p^{p-1}\in\mathbb{Z}_p^*$. This example gives rise to the crystalline cohomology theory.
    \item Let $(A,d)$ be $(\mathbb{Z}_p[[q-1]], [p]_q)$, where $[p]_q=\frac{q^p-1}{q-1}$ with $\delta$-structure given by $\phi(q)=q^p$. 
    Since $A$ is $(q-1)$-adically complete, an element of $A$ is a unit if and only if it is a unit modulo $(q-1)$. We have $\delta([p]_q)\equiv\delta(p)\mod{q-1}$, and thus the distinguishedness of $[p]_q$ follows from (i). This example leads to the $q$-de~Rham cohomology theory.
    \item Let $C/\mathbb{Q}_p$ be a perfectoid field and $(A, d)$ be $(A_{\inf}(\mathcal{O}_C), \xi)$, where $\xi$ is any generator of the kernel of Fontaine's map $A\to \mathcal{O}_C$. Then $A$ admits a unique $\delta$-structure induced by the Witt vector Frobenius, which is precisely the lift of Frobenius. The distinguishedness of $\xi$ can be seen similarly to part (ii). This construction gives the $A_{\inf}$-cohomology theory. 
\end{enumerate}
\end{Ex}

\begin{Def}[Category of prisms]
Fix a pair $(A,I)$ consisting of a $\delta$-ring $A$ and an ideal $I \subset A$. The collection of all such pairs forms a category called the category of $\delta$-pairs.
\begin{enumerate}[(i)]
\item A $\delta$-pair $(A, I)$ is a prism if $I\subset A$ is an invertible ideal such that $A$ is derived $(p, I)$-complete and $p\in IA+\phi(I)A$.
\item A map $(A, I)\to (B, J)$ is called (faithfully) flat if the map $A\to B$ is $(p,I)$-com\-pletely (faithfully) flat, i.e., $A/(p, I)\to B\otimes_A^L A/(p,I)$ is (faithfully) flat.
\item A prism $(A,I)$ is called
\begin{itemize}
    \item bounded if $A/I$ has bounded $p^\infty$-torsion, i.e., $A/I[p^\infty]=A/I[p^c]$ for some $c\geq 0$.
    \item crystalline if $I=(p)$.
    \item orientable if the ideal $I$ is principal; the choice of a generator is called an orientation.
    \item perfect if $A$ is a perfect $\delta$-ring, i.e, $\phi$ is an isomorphism.
\end{itemize}
\end{enumerate}
\end{Def}

\begin{Rem}
By \cite[Lemma~3.1]{Prisms}, the condition $p\in IA+\phi(I)A$ is equivalent to the fact that $I$ is pro-Zariski locally on $\Spec(A)$ generated by a distinguished element. Therefore, it is usually harmless to assume that $I=(d)$, i.e., that $(A, I)$ is orientable.
\end{Rem}

\begin{Ex}
Let $A$ be a $p$-torsion-free and $p$-complete $\delta$-ring. Then the pair $(A, (p))$ is a crystalline prism. Conversely, any crystalline prism arises in this way.
\end{Ex}

\begin{Ex}
Let $A_0=\mathbb{Z}_{(p)}\{d, \delta(d)^{-1}\}$ be the displayed localization of the free $\delta$-ring on a variable $d$. We denote by $A$ the $(p, d)$-completion of $A_0$. Then the pair $(A, (d))$ is a bounded prism, and moreover, it is the universal oriented prism.
\end{Ex}

Let us now recall the relationship between perfectoid rings and perfect prisms, starting with the notion of the perfection of a prism.  

\begin{Lemma}
Let $(A,I)$ be a prism, and let $A_{\perf}=\colim_{\phi} A$ be the perfection of~$A$. Then $IA_{\perf}=(d)$ is generated by a distinguished element, both $d$ and $p$ are non-zero-divisors in $A_{\perf}$, and $A_{\perf}/d[p^\infty]=A_{\perf}/d[p]$. In particular, the derived $(p,I)$-completion $A_\infty$ of $A_{\perf}$ agrees with the classical one, and $(A_\infty, IA_\infty)$ is the universal perfect prism under $(A,I)$.
\end{Lemma}

\begin{proof}
See \cite[Lemma~3.9]{Prisms}.
\end{proof}

\begin{Th}\label{Cat}
The following two categories are equivalent:
\begin{itemize}
    \item The category of (integral) perfectoid rings $R$.
    \item The category of perfect prisms $(A,I)$.
\end{itemize}
The equivalence is given by the functors $R \mapsto (A_{\inf}(R),\ker(\theta))$ and $(A,I) \mapsto A/I$, where $A_{\inf}(R)=W(R^\flat)$, and $\theta \colon A_{\inf}(R) \to R$ is the Fontaine's map.
\end{Th}

\begin{proof}
See \cite[Theorem~3.10]{Prisms}.
\end{proof}

\subsection{The prismatic site}
Now we are ready to define the prismatic site of a smooth $A/I$-algebra $R$. In some sense, the definition is a mixed-characteristic analogue of a crystalline site.

\begin{Def}
Let $(A,I)$ be a bounded prism and $R$ be a $p$-completely smooth $A/I$-algebra. The prismatic site of $R$ relative to $A$, denoted $(R/A)_\Prism$, is the category whose objects are prisms $(B, IB)$ over $(A, I)$ together with an $A/I$-algebra map $R\to B/IB$. Morphisms in this category are defined in the obvious way. We write a typical object of this category as $(R\rightarrow
B/IB\leftarrow B)$ and display it as
\[
\xymatrix{ A\ar[d]\ar[rr] & & B \ar[d] \\ A/I\ar[r] & R\ar[r] & B/IB.}
\]
We endow $(R/A)_\Prism$ with the indiscrete topology, so all presheaves are sheaves automatically.
We also define a structure sheaf $\mathcal{O}_\Prism$ and a structure sheaf modulo $I$,
denoted by $\overline{\mathcal{O}}_\Prism$, as functors that send $(R\rightarrow B/IB\leftarrow B)\in(R/A)_\Prism$ to
$B$ and $B/IB$ respectively. Note that $\overline{\mathcal{O}}_\Prism\simeq
\mathcal{O}_\Prism/I\mathcal{O}_\Prism$.
\end{Def}

\begin{Rem}
Strictly speaking, the category defined above is the opposite of what should be called the prismatic site. However, we think that there will not be any confusion because of this abuse of notation. We also refer to covariant functors $\mathcal{O}_\Prism$ and $\overline{\mathcal{O}}_\Prism$ as sheaves on $(R/A)_\Prism$. 
\end{Rem}

\begin{Rem}
The site defined in \cite[Definition~4.1]{Prisms} has a different Grothendieck topology: a map $(A,I)\to(B,IB)$ in $(R/A)_\Prism$ is a flat cover if it is a faithfully flat map of prisms. While this choice of topology changes the resulting topos, it does not affect the prismatic or Hodge--Tate cohomology. This phenomenon is similar to the corresponding situation in crystalline cohomology, where one can define the crystalline cohomology of an affine scheme using the indiscrete topology on the crystalline site while still obtaining the correct crystalline cohomology groups. The topology defined by flat covers also gives better results when replacing the ring $R$ with a smooth $p$-adic formal scheme $X$. To avoid the technical complexities of the Grothendieck topology, we restrict our discussion in this paper to the affine case.
\end{Rem}

\begin{Ex}
Assume that $A/I=R$. Then, the category $(R/A)_\Prism$ identifies with the category of
prisms over $(A,I)$, and its initial object is $(R\simeq A/I\leftarrow A)$.
\end{Ex}

\begin{Def} The prismatic complex $\Prism_{R/A}$ of an algebra $R$ is defined to be $R\Gamma((R/A)_\Prism, \mathcal{O}_\Prism)$. 
This is a $(p,I)$-complete commutative algebra object in $D(A)$. Note that the Frobenius on 
$\mathcal{O}_\Prism$ induces a $\phi$-semi-linear map $\Prism_{R/A}\to\Prism_{R/A}$.

We also define the Hodge--Tate complex $\overline{\Prism}_{R/A}\coloneq R\Gamma((R/A)_\Prism, \overline{
\mathcal{O}}_\Prism)$. There is an obvious isomorphism $\Prism_{R/A}\otimes^L_A A/I\simeq
\overline{\Prism}_{R/A}$.
\end{Def}

\begin{Ex}
If $R=A/I$, then $\Prism_{R/A}\simeq A$ and $\overline{\Prism}_{R/A}\simeq A/I$. This follows immediately from the fact that $(R\simeq A/I\leftarrow A)\in(R/A)_\Prism$ is the initial object in $(R/A)_\Prism$.
\end{Ex}

For our subsequent work, we will need to know that the Hodge--Tate complex localizes for the \'etale topology.

\begin{Lemma}[\'Etale localization]\label{Etale}
Let $R \to S$ be a $p$-completely \'etale map of $p$\nobreakdash-com\-pletely smooth $A/I$-algebras. Then the natural map $\overline{\Prism}_{R/A} \widehat{\otimes}_R^L S \to \overline{\Prism}_{S/A}$ is an isomorphism.
\end{Lemma}

\begin{proof}
See \cite[Lemma~4.21]{Prisms}.
\end{proof}

\section{De~Rham--Witt forms}
As we mentioned in the introduction, if the statement of Question~\ref{main} is true, one requires that, for any orientable prism $(A, d)$ and any $n\geq 1$, there is a canonical map
\[
W_n(A/d)\to A/d\phi(d)\cdots\phi^{n-1}(d).
\]

The construction of these maps is contained in this section. The alternative construction is contained in Appendix~\ref{appendix}. We also present the necessary information on the de~Rham--Witt complex following \cite{Ill} and \cite{LZ}. The section ends with an application of the higher Cartier isomorphism to some form of the de~Rham--Witt comparison for crystalline prisms.

\subsection{The universal map}
Let $(A,d)$ be an oriented prism. For notational convenience, we denote $d\phi(d)\cdots\phi^{n-1}(d)$ by $\tilde{d}_n$. We begin with the following proposition:

\begin{Prop}\label{map}
For any integer $n\geq 1$, there exists a functorial map 
\[W_n(A/d)\to A/d\phi(d)\cdots\phi^{n-1}(d),
\]
which is an isomorphism when $(A, d)$ is perfect. Moreover, in this case, the inverse map can be explicitly written as $\Tilde{\theta}_n\circ\phi=\theta_n\circ\phi^{n-1}$, where $\theta_n$ and $\Tilde{\theta}_n$ are Fontaine's theta maps (see the discussion following \cite[Lemma~3.2]{BMS_I}).
\end{Prop}

\begin{proof}
First, we assume that $A/d$ is $p$-torsion-free. Then $(A, d)$ is a transversal prism (see Remark~\ref{Bras} for explanations). In particular, the natural map
\[
A/d\phi(d)\cdots\phi^{n-1}(d)=A/\tilde{d}_n\hookrightarrow\prod_{i=0}^{n-1}A/\phi^i(d)
\]
is injective. Indeed, since $A/d[p]=0$, the prism $(A, d)$ is bounded. Therefore, the derived $(p, d)$-completion is the same as the classical one. Then it is easy to see that $(p,d)$ is a regular sequence in $A$. Hence, $(p, \phi^t(d))$ is also regular. By induction, one sees that $\phi^{k+s}(d)\equiv p\cdot \text{unit} \mod{\phi^s(d)}$. This implies the required result.

Thus, to define the required map, we can define a map from Witt vectors $W_n(A/d)$ to $\prod_{i=0}^{n-1}A/\phi^i(d)$  and show that its image actually contained in $A/\tilde{d}_n$.

So we define a map $r_n\colon W_n(A/d)\to \prod_{i=0}^{n-1}A/\phi^i(d)$ on generators by

\[
V^j([x])\mapsto (p^j x^{p^{n-j-1}},p^j\phi(x)^{p^{n-j-2}},\ldots,p^j\phi^{n-j-1}(x),\underbrace{0\ldots,0}_{\text{$j$ times}}).
\]

We now treat the universal case, i.e., when $(A, d)$ is the universal oriented prism. In this situation the perfection map $A\to A_{\infty}$ is $(p,d)$-completely faithfully flat (for details see \cite[~Example~3.4 or the proof of Theorem~15.2]{Prisms}). Using $p$-completely faithfully flat descent, it will suffice to prove the proposition for $A_\infty$. Indeed, consider the commutative diagram of Amitsur complexes:
\[
\xymatrix{
 0\ar[r] & A/\tilde{d}_n\ar[r]\ar[d] & A_\infty/\tilde{d}_n\ar[r]\ar[d]  & (A_\infty\widehat{\otimes}_A A_\infty)/\tilde{d}_n\ar[d]\\
  0\ar[r] & \prod_{i=0}^{n-1}A/\phi^i(d)\ar[r] &\prod_{i=0}^{n-1}A_\infty/\phi^i(d)\ar[r] & \prod_{i=0}^{n-1}(A_\infty\widehat{\otimes}_A A_\infty)/\phi^i(d).}
\] 

First, by \cite[Lemma~3.3]{Ar}, the quotients $A/\tilde{d}_n$, $A_\infty/\tilde{d}_n$, $A/\phi^i(d)$, and $A_\infty/\phi^i(d)$ are $p$-torsion-free. Therefore, \cite[Proposition~1.9]{Tian}  implies that both rows of the diagram are exact.

Second, since $A\to A_\infty$ is $(p,d)$-completely faithfully flat and $A$ is transversal, it follows from \cite[Remark~2.4.4]{BHLURIE} that  $A_\infty$ is also transversal. Thus, by \cite[Lemma~3.6]{Ar}, the left and central vertical arrows must be injective. The right vertical map is also injective; this follows from $(p,d)$-completely flatness of $A\to A_\infty$ and the injectivity of the central vertical arrow (see \cite[Lemma~A.3]{Arr} or \cite[Proposition~1.9]{Tian}).

Thus, if some element $t$ lies in $\prod_{i=0}^{n-1}A/\phi^i(d)$ contained in $\prod_{i=0}^{n-1}A_\infty/\phi^i(d)$ and we know that $t$ also comes from $A_\infty/\tilde{d}_n$, then it must lie in $A/\tilde{d}_n$. Hence, we reduce to the case of $A_\infty$. From now on, let us denote this ring simply by $A$.

We recall the following result (see \cite[Example~3.16 and Lemma~3.3]{BMS_I}).

\begin{Lemma}\label{ghost}
Let $R$ be a perfectoid ring and $\theta_r\colon A_{\inf}(R)\to W_r(R)$ be Fontaine's map. Let $w_r \colon W_r(R)\to R^r$ be a ghost map. Then 
\[
w_r\circ\theta_r=(\theta, \theta\phi,\theta\phi^2,\ldots,\theta\phi^{r-1}).
\]
\end{Lemma}

Observe that by Theorem~\ref{Cat}, the construction $(A,I)\mapsto A/I$ defines an
equivalence of categories between perfect prisms and perfectoid rings. Moreover, $A$ can be recovered from $A/I$ as $A_{\inf}(A/I)$, and $I$ can be recovered as the kernel of Fontaine's map. Since in our case $(A,d)$ is a perfect prism, the ring $A/d$ is perfectoid and $A$ can be recovered from $A/d$ as $A_{\inf}(A/d)$. Hence, by \cite[ Lemma~3.12]{BMS_I}, we have the following isomorphisms:
\[
\xymatrix{
A/d\phi(d)\cdots\phi^{n-1}(d)\ar[r]^-{\phi^{-n+1}}&A/d\phi^{-1}(d)\ldots\phi^{-n+1}(d)\ar[r]^-{\theta_n}&  W_n(A/d).
}
\]

We denote the inverse map by $r'_n$. We will check that $r_n$ coincides with $r'_n$ on generators. First, for simplicity,  we show that $r'_n(V(1))=(p,p,\ldots,p,0)\in \prod A/\phi^i(d)$. In order to do this, we observe that since $A/d$ is $p$-torsion-free, it follows that the ghost map $w_n \colon W_n(A/d)\to (A/d)^n$ is injective. Let $t\in A/d\phi^{-1}(d)\ldots\phi^{-n+1}(d)$ be such an element that $\theta_n(t)=V(1)$ and $y\in A/d\phi(d)\cdots\phi^{n-1}(d)$ be such that $y=\phi^{n-1}(t)=r'_n(V(1))$. We know that $w_n(V(1))=(0,p,\ldots,p)$. Moreover, by Lemma~\ref{ghost}, we see that 
\[ w_n(V(1)) = w_n(\theta_n(t)) = (\theta(t), \theta(\phi(t)),\ldots,\theta(\phi^{n-1}(t))). \]
Note that in our case $\theta\colon A\to A/d$  is just a quotient map, so we can conclude that $t\equiv 0 \mod{d}$ and $\phi^i(t)\equiv p\mod{d}$ for $i=1,
\ldots,n-1$ and, in terms of $y$, we have $\phi^{-n+1}(y)\equiv 0\mod{d}$ and $\phi^{-n+k}(y)\equiv p \mod{d}$ for $k=2,\ldots, n$. Applying $\phi^{n-k}$ to both sides of the congruence, we see that $y\equiv p\mod{\phi^i(d)}$, for $i=0,\ldots, n-2$ and $y\equiv 0\mod{\phi^{n-1}(d)}$. We are done with this case.

By a similar computation, one checks that
\[r_n'( V^j([x]))=(p^j x^{p^{n-j-1}},p^j\phi(x)^{p^{n-j-2}},\ldots,p^j\phi^{n-j-1}(x),\underbrace{0\ldots,0}_{\text{$j$ times}})\in\prod_{i=0}^{n-1}A/\phi^i(d),
\]
thus the last term actually lies in $A/d\phi(d)\cdots\phi^{n-1}(d)$.

Indeed, take an element $t$, which maps to $V^j([x])$ under the map $\theta_n$, and take $y=\phi^{n-1}(t)=r'_n(V^j([x]))$. We know that 
\[
w_n(V^j([x]))=(\underbrace{0,0,\ldots,0}_{\text{$j$ times}},p^jx,p^jx^p,\ldots, p^jx^{p^{n-j-1}}).
\]
Using this identity and that
\[
w_n(V^j([x]))=w_n(\theta_n(t))=(\theta(t), \theta(\phi(t)),\ldots,\theta(\phi^{n-1}(t))),
\]
we see that $\phi^i(t)\equiv 0 \mod{d}$ for $i=0, \ldots, j-1$ and $\phi^i(t)\equiv p^jx^{p^{i-j}}\mod{d}$ for $i=j, \ldots, n-1$. Using that $y=\phi^{n-1}(t)$ and applying the corresponding power of $\phi$, we conclude that 
\[
y\equiv p^j\phi^i(x)^{p^{n-i-j-1}} \mod{\phi^i(d)}
\]
for $i=0, \ldots, n-j-1$ and $y\equiv 0 \mod{\phi^i(d)}$ for $i=n-j, \ldots, n-1$.

So the images of generators $V^j([x])$ under the map $r_n$ lies in $A/d\phi(d)\cdots\phi^{n-1}(d)$. Moreover, the compatibility with the ring structure for $r_n$ follows since the map $r_n$ coincides with $r'_n=(\theta_n\circ\phi^{n-1})^{-1}$ on generators, as we have shown above. We conclude that $r_n$ coincides with $r'_n$ and the result for the universal prism follows.

Note that from the faithfully flat base change, the result and the formula follow automatically for free $\delta$-rings over the universal prism. We now assume that $A/d$ has no $p$-torsion. We can take a part of its simplicial resolution, $\Tilde{\Tilde{A}}\rightrightarrows\Tilde{A}\to A$, where $\Tilde{A}$ and $\Tilde{\Tilde{A}}$ are free $\delta$-rings over the universal prism. Then one has the following commutative diagram:
\[
\xymatrix{
W_n(\Tilde{\Tilde{A}}/d) \ar[d] \ar@<3pt>[r] \ar@<-3pt>[r] & W_n(\Tilde{A}/d)\ar[d]\ar[r] & W_n(A/d) \ar@{-->}[d]\\
\Tilde{\Tilde{A}}/d\phi(d)\cdots \phi^{n-1}(d) \ar@<3pt>[r] \ar@<-3pt>[r] & \Tilde{A}/d\phi(d)\cdots\phi^{n-1}(d)\ar[r] & A/d\phi(d)\cdots\phi^{n-1}(d),}
\] 
where the right downward arrow exists by the universal property since $W_n(A/d)$ is a coequalizer of $W_n(\Tilde{\Tilde{A}}/d)\rightrightarrows W_n(\Tilde{A}/d)$. The explicit formula for a $p$-torsion-free case follows from the above and the surjectivity of right horizontal maps.  Therefore, we get the result and the explicit formula for arbitrary prism $(A,d)$ such that $A/d$ is $p$-torsion-free. 

To prove the result for a general prism $(A,d)$, we repeat the argument from the above paragraph: take a part of a simplicial resolution, $\Tilde{\Tilde{A}}\rightrightarrows\Tilde{A}\to A$, where $\Tilde{A}$ and $\Tilde{\Tilde{A}}$ are free $\delta$-rings over the universal prism. Note that, in the general case there is no way to get the explicit formula, since the map $A/d\phi(d)\cdots\phi^{n-1}(d)\to \prod_{i=0}^{n-1}A/\phi^i(d)$ is not injective anymore.  Again, consider the following commutative diagram
\[
\xymatrix{
W_n(\Tilde{\Tilde{A}}/d) \ar[d] \ar@<3pt>[r] \ar@<-3pt>[r] & W_n(\Tilde{A}/d)\ar[d]\ar[r] & W_n(A/d) \ar@{-->}[d]\\
\Tilde{\Tilde{A}}/d\phi(d)\cdots \phi^{n-1}(d) \ar@<3pt>[r] \ar@<-3pt>[r] & \Tilde{A}/d\phi(d)\cdots\phi^{n-1}(d)\ar[r] & A/d\phi(d)\cdots\phi^{n-1}(d)}
\] 
and conclude because $W_n(A/d)$ is a coequalizer of $W_n(\Tilde{\Tilde{A}}/d)\rightrightarrows W_n(\Tilde{A}/d)$.
\end{proof}

\begin{Rem}\label{Bras}
One of the key points of the above proof was the inclusion
\[
A/d\phi(d)\cdots\phi^{n-1}(d) \hookrightarrow \prod_{i=0}^{n-1}A/\phi^i(d)
\]
for a prism of special type. Recall, (see \cite[Definition~3.2]{Ar}) that a prism $(A, d)$ is called transversal if $(p, d)$ is a regular sequence in $A$. Several useful technical results are known for such prisms. In particular, in \cite[Lemma~3.6]{Ar}, the desired inclusion is proved for transversal prisms. By our assumptions (in the part of the proof where $A/d$ is $p$-torsion-free), the prism is transversal, and therefore the results from~\cite{Ar} can be used. Indeed, in our case $(d, p)$ is a regular sequence, and the prism is bounded. Hence, the derived completion coincides with the classical one. Then it is well-known that, for any ring $R$ with a regular sequence $(r, s)$ such that $R$ is $r$-adically complete, the sequence $(s, r)$ is also regular. For additional information, see \cite[Section~3]{Ar}.
\end{Rem}

\begin{Rem}\label{rem2}
From the construction of the map above and \cite[Lemma~3.4]{BMS_I}, we see that the following diagram is commutative:
\[\xymatrix{
W_{r+1}(A/d)\ar[r]^-{\lambda_{r+1}}  & A/d\phi(d)\cdots\phi^r(d) \\
W_r(A/d)\ar[r]^-{\lambda_r}\ar[u]^{V} & A/d\phi(d)\cdots\phi^{r-1}(d)\ar[u]_{\phi^r(d)\times\text{unit
}},}
\]
where $\lambda_r$ is the map defined in Proposition~\ref{map}.
\end{Rem}

\begin{Rem}\label{etale} There is an \'etale localization property for $\Prism_{R/A}/\tilde{d}_n$. Indeed, with the notation of Lemma~\ref{Etale}, one has the following isomorphism:
\[\Prism_{R/A}/\tilde{d}_n\widehat{\otimes}^L_{W_n(R)}W_n(S)\simeq \Prism_{S/A}/\tilde{d}_n.
\]
To see this, by derived Nakayama argument, it is enough to check the isomorphism after a base change $-\widehat{\otimes}^L_{A/\tilde{d}_n}A/d$. Hence, it suffices to show that
\[\overline{\Prism}_{R/A}\widehat{\otimes}^L_{W_n(R)}W_n(S)\simeq \overline{\Prism}_{S/A}.
\]
Next, the left-hand side can be rewritten as $\overline{\Prism}_{R/A}\widehat{\otimes}^L_{R}(R\widehat{\otimes}^L_{W_n(R)}W_n(S))$. Since $R\to S$ is $p$-completely \'etale, the term in the last brackets is exactly $S$ (see~\cite[Theorem~10.4]{BMS_I} or \cite{LZ}). Hence, we conclude the result by Lemma~\ref{Etale}.
\end{Rem}

\subsection{The de~Rham--Witt complex} In this subsection, we introduce basic definitions and necessary results for the relative de~Rham--Witt complex, following~\cite{LZ} and~\cite{BMS_I}. Let $A$ be a $\mathbb{Z}_{(p)}$-algebra.

\begin{Def} Let $B$ be an $A$-algebra. An $F$-$V$-procomplex for $B/A$ consists of the following data $(\mathcal{W}_r^\bullet, R, F, V, \lambda_r)$:
\begin{enumerate}[(i)]
    \item a commutative differential graded $W_r(A)$-algebra $\mathcal{W}_r^\bullet=\bigoplus_{n\geq 0} \mathcal{W}_r^n$ for each integer $r\geq 1$;
    \item morphisms $R\colon \mathcal{W}_{r+1}^\bullet\to R_*\mathcal{W}_r^\bullet$ of differential graded $W_{r+1}(A)$-algebras for $r\geq 1$;
    \item morphisms $F\colon \mathcal{W}_{r+1}^\bullet\to F_*\mathcal{W}_r^\bullet$ of graded $W_{r+1}(A)$-algebras for $r\geq1$;
    \item morphisms $V\colon F_*\mathcal{W}_r^\bullet\to \mathcal{W}_{r+1}^\bullet$ of graded $W_{r+1}(A)$-modules for $r\geq1$;
    \item morphisms $\lambda_r\colon W_r(B)\to \mathcal{W}_r^0$ for each $r\geq 1$, commuting with $R,V,F$;   
\end{enumerate}
such that the following identities hold:
\begin{itemize}
    \item[-] $R$ commutes with both $F$ and $V$;
    \item[-] $FV$ is a multiplication by $p$;
    \item[-] $FdV=d$;
    \item[-] $V(F(x)y)=xV(y)$;
    \item[-] (Teichm\"uller identity). $Fd\lambda_{r+1}([b])=\lambda_r([b])^{p-1}d\lambda_r([b])$ for $b\in B$ and $r\geq 1$.
\end{itemize}

These complexes are also known as Witt complexes.
\end{Def}

\begin{Th}[see \cite{LZ}]
There is an initial object $\{W_r\Omega^\bullet_{B/A}\}_r$ in the category of $F$-$V$-procomplexes for $B/A$, called the relative de~Rham--Witt complex. This means that, if $(\mathcal{W}_r, R,F,V,\lambda_r)$ is any $F$-$V$-procomplex for $B/A$, then there are unique maps of graded $W_r(A)$-algebras
\[\lambda_r^\bullet\colon W_r\Omega^\bullet_{B/A}\to\mathcal{W}_r^\bullet,
\]
which are compatible with $R,F,V$ in the obvious sense and such that
\[\lambda_r^0\colon W_r(B)\to\mathcal{W}_r^0
\]
is the structure map $\lambda_r$ of the Witt complex $\mathcal{W}_r$ for any $r\geq 1$. 
\end{Th}

\begin{sketch}[adapted from \cite{Ill}] We need to check the following two key properties:
\begin{enumerate}[(i)]
    \item For any $F$-$V$-procomplex $\mathcal{W}^\bullet$, and all $n\geq 1$, the map $d\colon W_n(B)\to \mathcal{W}^1_n$ is a pd-derivation. Hence, $\mathcal{W}_n$ is a pd-differential graded algebra. To check that the map $d$ is a pd-derivation, we need to show that, for any $x\in B$,
    \[
    d\gamma_p(V[x])=\gamma_{p-1}(V[x])dV[x].
    \]
   By the definition of the pd-structure on $W(B)$, this equality holds if and only if $p^{p-2}dV[x]^p=p^{p-2}V[x]^{p-1}dV[x]$, but already
    \[dV[x]^p=d([x]V(1))=V(1)d[x]=VFd[x]=V([x]^{p-1}d[x])=V[x]^{p-1}dV[x].
    \]
    \item If $D\colon W_n(A)\to M$ is a pd-derivation into $W_n(A)$-module $M$, then the map $FD\colon W_{n-1}A\to F_*M$ defined by
    \[
    FDx=[a^{p-1}]D[a]+DV[b]
    \]
    for $x=[a]+V[b]$ is also a pd-derivation.
\end{enumerate}
In particular, it follows from property~(ii) that the projective system of pd-differen\-tial de~Rham complexes $\breve{\Omega}^\bullet_{W_n(B)/W_n(A)}$ naturally acquires maps of graded algebras
\[
F\colon \breve{\Omega}^\bullet_{W_n(B)/W_n(A)}\to \breve{\Omega}^\bullet_{W_{n-1}(B)/W_{n-1}(A)},
\]
which satisfy the following identities from the definition of $F$-$V$-procomplex: 
\begin{center}
($FdVx=dx$ for $x\in W_n(B)$, $Fd[x]=[x^{p-1}]d[x]$ for $x\in B$, $dFx=pFdx$).
\end{center}
The relative de~Rham--Witt complex $W_n \Omega^\bullet_{B/A}$ is then constructed inductively as a quotient of $\breve{\Omega}^\bullet_{W_n(B)/W_n(A)}$.
\end{sketch}

\begin{Def}
    Let $A\to B$ be a morphism of $\mathbb{Z}_{(p)}$-algebras.
    The continuous de~Rham--Witt complex of $B$ over $A$ of length $r$ is defined as
    \[
    W_r\Omega^{i, cont}_{B/A}=\varprojlim_{s}W_r\Omega^i_{(B/p^s)/(A/p^s)}=\varprojlim_{s}W_r\Omega^i_{B/A}/p^s.
    \]
\end{Def}

\begin{Th}[Higher Cartier isomorphism]\label{HighCart}
Let $X/k$ be a smooth scheme over a perfect field $k$ of characteristic $p>0$. Then, for any integer $n\geq 1$, the map $F^n\colon W_{2n}\Omega^i_{X/k}\to W_n\Omega^i_{X/k}$ induces an isomorphism
\[W_n\Omega^i_{X/k}\to\mathcal{H}^iW_n\Omega^\bullet_{X/k},
\]
which is compatible with products and equals to the inverse of the Cartier operator~$C^{-1}$ for $n=1$.
\end{Th}

\begin{proof} The key point is to establish   $F^nW_{2n}\Omega^i_{X/k}=ZW_n\Omega^i_{X/k}$. The original proof given in \cite{Il} was incomplete and was subsequently corrected in \cite{IlR}. The argument relies on the explicit description of $W_n\Omega^\bullet_{X/k}$ for $X=\Spec(k[t_1,\ldots, t_n])$ in terms of the complex of integral forms and the properties of the Cartier isomorphism.
\end{proof}

\begin{Lemma}[\'Etale base change]\label{Deter}
Let $A\to B$ be a map of $\mathbb{Z}_{(p)}$-algebras, and let $S$ be an \'etale $B$-algebra. Then the natural map
\[W_r\Omega^n_{B/A}\otimes_{W_r(B)}W_r(S)\to W_r\Omega^n_{S/A}
\]
is an isomorphism.
\end{Lemma}

\begin{proof}
If $p$ is nilpotent in $B$ or $B$ is $F$-finite, the result follows from \cite[Proposition~1.7]{LZ}. This assumption is used in \cite{LZ} only to ensure that the map $W_r(B) \to W_r(S)$ is \'etale. However, by \cite[Theorem~10.4]{BMS_I}, this map is always \'etale. Therefore, the argument in \cite{LZ} applies in general.
\end{proof}

Let $S$ be a ring such that $p$ is nilpotent in $S$, and let $X$ be a smooth scheme over~$S$. Denote by $X_{zar}$ the topos of Zariski sheaves on $X$. For each positive integer~$n$, let $(X/W_n(S))_{crys}$ denote the crystalline topos, and let $u_n\colon (X/W_n(S))_{crys}\to X_{zar}$ be the natural map of topoi.

An important advantage of the de~Rham--Witt complex is the existence of a comparison isomorphism with crystalline cohomology.

\begin{Th}\label{cris}
Let $S$ be a ring such that $p$ is nilpotent in $S$, and let $X$ be a smooth scheme over $S$. There exists a canonical map of projective systems in the derived category $D^+(X, W_n(S))$ of sheaves of $W_n(S)$-modules on $X_{zar}$:
\[Ru_{n*}\mathcal{O}_{X/W_n(S)}\to W_n\Omega^\bullet_{X/S},
\]
which is, in fact, an isomorphism. 
\end{Th}
\begin{proof}
For the special case where $S=\Spec(k)$ with $k$ a perfect field of characteristic~$p$, the proof was given in \cite{Il}. The general case was obtained in \cite{LZ}.
\end{proof}

\subsection{The de~Rham--Witt complex of a polynomial algebra}
We now recall several results of Langer--Zink concerning the relative de~Rham--Witt complex of $A[\underline{T}]=A[T_1,\ldots, T_d]$ and $A[\underline{T}^{\pm 1}]=A[T_1^{\pm 1},\ldots, T_d^{\pm 1}]$.

Let $A$ be a $\mathbb{Z}_{(p)}$-algebra. In this subsection, we provide an explicit description of the relative de~Rham--Witt complex of the Laurent polynomial algebra $A[\underline{T}^{\pm 1}]$. Let $a\colon \{1,\ldots,d\}\to p^{-r}\mathbb{Z}$ be a weight function. We define $\nu(a)\coloneq \min_{i}{\nu(a(i))}$, where $\nu(a(i))=\nu_p(a(i))\in\mathbb{Z}\cup\{\infty\}$ is the $p$-adic valuation of $a(i)$. For any subset $I\subset\{1, \ldots, d\}$, we also set $\nu(a|_I)\coloneq \min_{i\in I}\nu(a(i))$. 

We denote by $P_a$ the collection of disjoint partitions $I_0,\ldots, I_n$ of $\{1,\ldots,d\}$ satisfying:

\begin{enumerate}[(i)]
    \item each subset $I_1,\ldots, I_n$ must be non-empty, while $I_0$ may be empty;
    \item the $p$-adic valuation of any element in $a(I_{j-1})$ is at most the $p$-adic valuation of any element in $a(I_j)$, for $j=1,\dots,n$;
    \item to resolve ambiguities when $\nu\colon\{1,\ldots,d\}\to\mathbb{Z}$ is non-injective, we fix a total ordering $\preceq_a$ on $\{1,\ldots,d\}$ such that $\nu$ is weakly increasing. We then require that every element of $I_{j-1}$ is strictly $\preceq_a$-less than all elements of $I_j$.
\end{enumerate}

For a partition $(I_0,\ldots, I_d)\in P_a$, define $\rho_1$ as the greatest integer between $0$ and $n$ such that $\nu(a|_{I_{\rho_1}})<0$ (set $\rho_1=0$ if no such integer exists). Similarly, let $\rho_2$ be the greatest integer between $0$ and $n$ such that $\nu(a|_{I_{\rho_2}})<\infty$. 

For convenience, we define $u(a):=\max\{0, -\nu(a)\}$. Given $x\in W_{r-u(a)}(A)$, we construct an element $e(x, a, I_0,\ldots, I_n)\in W_r\Omega^n_{A[\underline{T}^{\pm1}]/A}$ as follows:

\begin{enumerate}[(i)]
    \item $(I_0\neq\emptyset)$ the product of elements
    \[V^{-\nu(a|_{I_0})}(x\prod_{i \in I_0}[T_i]^{a(i)/p^{\nu(a|_{I_0})}})
    \]
   \[dV^{-\nu(a|_{I_j})}\prod_{i \in I_j}[T_i]^{a(i)/p^{\nu(a|_{I_j})}}\text{, where $j=1,\ldots,\rho_1$},
    \]
      \[F^{\nu(a|_{I_j})}d\prod_{i \in I_j}[T_i]^{a(i)/p^{\nu(a|_{I_j})}}\text{, where $j=\rho_1+1,\ldots,\rho_2$},
    \]
    \[d\log\prod_{i\in I_j}[T_i] \text{, where $j=\rho_2+1,\ldots, n$}.
    \]
    \item $(I_0=\emptyset, \nu(a)<0)$ the product of elements
    \[dV^{-\nu(a|_{I_1})}(x\prod_{i \in I_1}[T_i]^{a(i)/p^{\nu(a|_{I_1})}})
    \]
   \[dV^{-\nu(a|_{I_j})}\prod_{i \in I_j}[T_i]^{a(i)/p^{\nu(a|_{I_j})}}\text{, where $j=2,\ldots,\rho_1$},
    \]
    \[F^{\nu(a|_{I_j})}d\prod_{i \in I_j}[T_i]^{a(i)/p^{\nu(a|_{I_j})}}\text{, where $j=\rho_1+1,\ldots,\rho_2$},
    \]
    \[d\log\prod_{i\in I_j}[T_i] \text{, where $j=\rho_2+1,\ldots, n$}.
    \]
    \item $(I_0=\emptyset, \nu(a)\geq0)$ the product of $x\in W_r(A)$ with the elements
   
    \[F^{\nu(a|_{I_j})}d\prod_{i \in I_j}[T_i]^{a(i)/p^{\nu(a|_{I_j})}}\text{, where $j=1,\ldots,\rho_2$},
    \]
    \[d\log\prod_{i\in I_j}[T_i] \text{, where $j=\rho_2+1,\ldots, n$}.
    \]
\end{enumerate}

\begin{Th}[{\cite[Proposition~2.17]{LZ}}]\label{poly} The map of $W_r(A)$-modules
\[e\colon\bigoplus_{a:\{1,\ldots, d\}\to p^{-r}\mathbb{Z}}\bigoplus_{(I_0,\ldots, I_n)\in P_a}V^{u(a)}W_{r-u(a)}(A)\to W_r\Omega^n_{A[\underline{T}^{\pm 1}]/A}
\]
given by $V^{u(a)}(x)\mapsto e(x,a, I_0,\ldots, I_n)$ is an isomorphism.
\end{Th}

\begin{proof}
See \cite[Theorem~10.12]{BMS_I}.
\end{proof}

\begin{Rem}
To describe the de~Rham--Witt complex for a polynomial algebra $A[\underline{T}]$ (rather than $A[\underline{T}^{\pm 1}]$), replace $p^{-r}\mathbb{Z}$ with $p^{-r}\mathbb{Z}_{\geq 0}$.
\end{Rem}

\subsection{An application of the higher Cartier isomorphism to special perfect prisms}
Our main goal is to prove Theorem~\ref{perfect_case}.
We will present its proof in the next section, but first we consider a special case where $d=p$ and $A/p$ is a perfect field.

\begin{Th}
Let $(A,p)$ be a crystalline prism such that $A/p$ is a perfect field, and let $S$ be a smooth $A/p$-algebra. Then there exists an isomorphism
\[
W_r\Omega^i_{S/(A/p)}\simeq H^i((S/A)_\Prism,\mathcal{O}_\Prism/p^r).
\]
\end{Th}

\begin{proof}
Let us denote the field $A/p$ by $k$. Our argument is analogous to the proof of the Hodge--Tate comparison in characteristic~$p$ (see \cite[Corollary~5.5]{Prisms}). By \'etale localization, we can assume that $S=k[T_1,\ldots, T_d]$. By the higher Cartier isomorphism (Theorem~\ref{HighCart}), we have $W_r\Omega^i_{S/k}\simeq H^i(W_r\Omega^\bullet_{S/k})$. Furthermore, by Theorem~\ref{cris}, $H^i(W_r\Omega^\bullet_{S/k})\simeq H^i(Ru_{r*}\mathcal{O}_{S/W_r(k)})$. Then, from the crystalline comparison for prismatic cohomology (see \cite[Theorem~5.2]{Prisms}), it is known that
\[
\begin{aligned}
H^i(Ru_{r*}\mathcal{O}_{S/W_r(k)}) &\simeq H^i(Ru_{r*}\mathcal{O}_{S/W(k)}\otimes_{W(k)}^L W_r(k)) \\
&\simeq H^i(\phi^*\Prism_{S^{(1)}/W(k)}\otimes_{W(k)}^L W_r(k)).
\end{aligned}
\]

Since $S=k[T_1,\ldots, T_n]$ is a polynomial algebra over the perfect field $k$, we have $\phi^*\Prism_{S^{(1)}/W(k)}\simeq\Prism_{S/W(k)}$. Therefore, combining these isomorphisms, we see that \[W_r\Omega^i_{S/k}\simeq H^i(\Prism_{S/W(k)}\otimes_{W(k)}^L W_r(k))\simeq H^i((S/A)_\Prism,\mathcal{O}_\Prism/p^r).
\]
\end{proof}

\section{The de~Rham--Witt comparison}
In this section, we prove Theorem~\ref{perfect_case} and provide an explicit description of prismatic cohomology for a $p$-completed polynomial algebra. Together they constitute the main results of the paper.  

We begin by discussing the connection established in \cite{Prisms} between prismatic cohomology and the results presented in \cite{BMS_I}. 
Let us fix a perfectoid field $C$ of characteristic~$0$ containing  $\mu_{p^\infty}$. Let $R$ be a $p$-completely smooth $\mathcal{O}_C$-algebra.  In \cite{BMS_I}, the authors  constructed a complex $A\Omega$ that relates to prismatic cohomology through the following theorem:

\begin{Th}\label{BMScomp} There is an isomorphism 
\[A\Omega_{R/A}\simeq \Prism_{R^{(1)}/A}=\phi^*_A\Prism_{R/A}
\]
of $E_\infty$-$A$-algebras that is compatible with the Frobenius.
\end{Th}

\begin{proof}
See \cite[Theorem~17.2]{Prisms}.
\end{proof}

This isomorphism enables us to prove Theorem~\ref{perfect_case} in a special case. First, recall the following sequence of isomorphisms established in \cite[Theorem~9.2]{BMS_I}:
\[H^i(A\Omega_{R/\mathcal{O}_C}\otimes^L_{A_{\inf}(\mathcal{O}_C),\Tilde{\theta}_r}W_r(\mathcal{O}_C))\simeq H^i(\widetilde{W_r\Omega}_{R/\mathcal{O}_C})\simeq W_r\Omega^{i, cont}_{R/\mathcal{O}_C}\{-i\}.
\]
From Theorem~\ref{BMScomp} and Proposition~\ref{map}, we derive that 
\[H^i(\Prism_{R/A}\otimes^L_A A/d\phi(d)\cdots\phi^{r-1}(d))\simeq H^i(A\Omega_{R/\mathcal{O}_C}\otimes^L_{A_{\inf}(\mathcal{O}_C),\Tilde{\theta}_r}W_r(\mathcal{O}_C)).
\]
It is important to note that the map constructed in Proposition~\ref{map} is inverse to $\Tilde{\theta}_r$ up to $\phi^{-1}$. Thus, in this specific case, the statement of Question~\ref{main} holds true. For notational convenience, we shall henceforth omit the Breuil--Kisin twist in the notation of de~Rham--Witt forms.

The proof of Theorem~\ref{perfect_case} relies fundamentally on Andr\'e's lemma, which was originally established in \cite{An} and later reproved in \cite{Prisms}.

\begin{Lemma}[Andr\'e's flatness lemma] Let $R$ be a perfectoid ring. Then there exists a $p$-completely faithfully flat map $R\to S$ of perfectoid rings such that $S$ is absolutely integrally closed. In particular, every element of $S$ admits a compatible system of $p$-power roots.
\end{Lemma}
\begin{proof}
See \cite[Theorem~7.14]{Prisms}.
\end{proof}

We now state the main result of this paper:

\begin{Th}\label{Proof}
Let $(A,d)$ be a perfect prism such that $A/d$ is $p$-torsion-free, and let $S$ be a $p$-completely smooth $A/d$-algebra. Then there exists a functorial isomorphism
\[W_n\Omega_{S/(A/d)}^{i, cont}\to H^i(\Prism_{S/A}\otimes^L_A A/d\phi(d)\cdots\phi^{n-1}(d)).
\]
\end{Th}

\begin{Rem}\label{unique}
Following the approach in \cite[Construction~4.9 and the discussion after Lemma~4.10]{Prisms}, one can verify that $H^*(\Prism_{S/A}\otimes^L_A A/d\phi(d)\cdots\phi^{n-1}(d))$ is a differential graded $A/d\phi(d)\cdots\phi^{n-1}(d)$-algebra with the Bockstein homomorphism as its differential. This algebra is also graded commutative by generalities on cohomology of commutative ring objects in a topos. Moreover, it is clear that when $p\neq 2$, this dga is strictly commutative.

If we denote by $\eta\colon W_n(S)\to H^0(\Prism_{S/A}\otimes^L_A A/d\phi(d)\cdots\phi^{n-1}(d))$ the structure morphism, then, by the universal property of the de~Rham complex, $\eta$ extends uniquely to a map
\[\Omega^*_{W_n(S)/W_n(A/d)}\to H^*(\Prism_{S/A}\otimes^L_A A/d\phi(d)\cdots\phi^{n-1}(d))\]
of $W_n(A/d)$-dgas.

Since $W_n\Omega^*_{S/(A/d)}$ is a quotient of $\Omega^*_{W_n(S)/W_n(A/d)}$, we can conclude that the desired map from $W_n\Omega^*_{S/(A/d)}$ to  $H^*(\Prism_{S/A}\otimes^L_A A/d\phi(d)\cdots\phi^{n-1}(d))$ is unique if it exists. Since the last term is $p$-completed, the same conclusion applies to $W_n\Omega^{*, cont}_{S/(A/d)}$. This observation will play a crucial role in the proof of Theorem~\ref{Proof}.
\end{Rem}

\begin{proof} Our proof proceeds through several reduction steps.

Since $(A,d)$ is a perfect prism by assumption, $R\coloneq A/d$ is a perfectoid ring. By Andr\'e's lemma, there exists a $p$-completely faithfully flat map $R\to\Tilde{R}$, where $\Tilde{R}$ is absolutely integrally closed. In particular, $\Tilde{R}$ has a compatible system of $p$-power roots of unity, which induces a map $\hat{\overline{\mathbb{Z}}}_p\to\Tilde{R}$. We first demonstrate how to derive the comparison isomorphism for $\Tilde{R}$ from the comparison for $\hat{\overline{\mathbb{Z}}}_p$ that was elaborated at the beginning of this section. We recall the following well-known lemma from homological algebra:

\begin{Lemma}\label{superlemma}
Let $P\to Q$ be any ring homomorphism and $C\in \mathcal{D}^{-}(P)$. If
\[
H^j(C)\otimes^L_P Q\simeq H^j(C)\otimes_P Q
\]
for all $j$, then $H^j(C\otimes_P^L Q)\simeq H^j(C)\otimes_PQ$.
\end{Lemma}
\begin{proof}
    By \stacktag{0662}, there exists a spectral sequence with
    \[
    E_2^{i, j}=H^i(H^j(C)\otimes_P^L Q)\Rightarrow H^{i+j}(C\otimes_P^L Q).
    \]
    Since $H^i(H^j(C)\otimes_P^L Q)\simeq \Tor_i^P(H^j(C),Q)$, the isomorphism 
    \[
    H^j(C)\otimes^L_P Q\simeq H^j(C)\otimes_P Q
    \]
    implies that  $E_2^{i,j}$ vanishes unless $i=0$. Consequently, this spectral sequence yields the required isomorphism. 
\end{proof}

\begin{Rem}\label{Completed}
A parallel argument applied to the completed tensor product instead of the ordinary one shows that an analogous version of Lemma~\ref{superlemma} holds for completed base change.
\end{Rem}

We now reduce the problem to the case of the $p$-completed polynomial algebra $\Tilde{R}\langle T_1,\ldots, T_k\rangle$. Let  $S'$ be a $p$-completely smooth $\tilde{R}$-algebra. Then there exists a covering of $\Spf(S')$ by affine open subsets $U_i=\Spf(A_i)$ such that for each $A_i$, there is a $p$-completely \'etale map $\Tilde{R}\langle T_1,\ldots, T_{k_i}\rangle \to A_i$. By applying \'etale  localization  to both the relative de~Rham--Witt complex (see Lemma~\ref{Deter}) and the prismatic complex (see Remark~\ref{etale}), it suffices to prove the result for a $p$-completed polynomial algebra to establish the required comparison for each $A_i$.

Thus, we can assume that $\Spf(S')$ is covered by open sets $U_i$ for which the required isomorphism holds. When we restrict these isomorphisms for $U_i$ and $U_j$ to the intersection $U_i\cap U_j=\Spf(S'_{ij})$, we obtain two maps (which are isomorphisms) from the de~Rham--Witt forms $W_n\Omega_{S'_{ij}/\tilde{R}}^{i, cont}$ to $H^i(\Prism_{S'_{ij}/A_{\inf}(\tilde{R})}\otimes^L_{A_{\inf}(\tilde{R})}A_{\inf}(\tilde{R})/d\cdots\phi^{n-1}(d))$. By Remark~\ref{unique}, such a map must be unique, hence, these two isomorphisms coincide. Therefore, the gluing  condition is satisfied, and we automatically obtain the required isomorphism  as a map of sheaves on $\Spf(S')$. Moreover, by the local-to-global spectral sequence and the vanishing of cohomology of de~Rham--Witt forms on affines, this isomorphism also holds in terms of modules. Thus, we obtain  the required statement for $S'$, assuming the result for $\tilde{R}\langle T \rangle=\Tilde{R}\langle T_1,\ldots, T_k\rangle$.

Now we return to the map $\hat{\overline{\mathbb{Z}}}_p\to\tilde{R}$ and explain how to deduce the comparison for $\tilde{R}$ from the isomorphism for $\hat{\overline{\mathbb{Z}}}_p$. Let $C$ denote  $\Prism_{\hat{\overline{\mathbb{Z}}}_p\langle T\rangle /A_{\inf}(\hat{\overline{\mathbb{Z}}}_p)}\otimes^L_{A_{\inf}(\hat{\overline{\mathbb{Z}}}_p)}W_n(\hat{\overline{\mathbb{Z}}}_p)$. Assuming the comparison holds for $\hat{\overline{\mathbb{Z}}}_p$, we have $H^i(C)\simeq W_n\Omega^{i, cont}_{\hat{\overline{\mathbb{Z}}}_p\langle T\rangle /\hat{\overline{\mathbb{Z}}}_p}$. Since $\tilde{R}$ is perfectoid, by \cite[Proposition 10.14]{BMS_I}, we see that
\[
W_n\Omega^{i, cont}_{\hat{\overline{\mathbb{Z}}}_p\langle T\rangle/\hat{\overline{\mathbb{Z}}}_p}\widehat{\otimes}^L_{W_n(\hat{\overline{\mathbb{Z}}}_p)} W_n(\Tilde{R})\simeq W_n\Omega^{i, cont}_{\hat{\overline{\mathbb{Z}}}_p\langle T\rangle /\hat{\overline{\mathbb{Z}}}_p}\widehat{\otimes}_{W_n(\hat{\overline{\mathbb{Z}}}_p)} W_n(\Tilde{R})\simeq W_n\Omega^{i, cont}_{\Tilde{R}\langle T\rangle/\Tilde{R}}.
\]
Consequently,  by applying Lemma~\ref{superlemma} and Remark~\ref{Completed}, we obtain the second isomorphism in the following sequence:
\[ W_n\Omega^{i, cont}_{\Tilde{R}\langle T\rangle/\Tilde{R}}\simeq W_n\Omega^{i, cont}_{\hat{\overline{\mathbb{Z}}}_p\langle T\rangle /\hat{\overline{\mathbb{Z}}}_p}\widehat{\otimes}_{W_n(\hat{\overline{\mathbb{Z}}}_p)}W_n(\Tilde{R})\simeq H^i(C\widehat{\otimes}^L_{W_n(\hat{\overline{\mathbb{Z}}}_p)}W_n(\Tilde{R})).
\]
The last term is precisely 
$H^i(\Prism_{\Tilde{R}\langle T\rangle/A_{\inf}(\Tilde{R})}\otimes^L_{A_{\inf}(\Tilde{R})} W_n(\Tilde{R}))$, since
\[
H^i(\Prism_{\hat{\overline{\mathbb{Z}}}_p\langle T\rangle /A_{\inf}(\hat{\overline{\mathbb{Z}}}_p)}\otimes^L_{A_{\inf}(\hat{\overline{\mathbb{Z}}}_p)}W_n(\hat{\overline{\mathbb{Z}}}_p)\widehat{\otimes}^L_{W_n(\hat{\overline{\mathbb{Z}}}_p)}W_n(\Tilde{R}))
\]
is exactly
\[
H^i(\Prism_{\hat{\overline{\mathbb{Z}}}_p\langle T\rangle /A_{\inf}(\hat{\overline{\mathbb{Z}}}_p)}\widehat{\otimes}^L_{A_{\inf}(\hat{\overline{\mathbb{Z}}}_p)}W_n(\Tilde{R}))\simeq H^i(\Prism_{\Tilde{R}\langle T \rangle/A_{\inf}(\tilde{R})}\otimes^L_{A_{\inf}(\tilde{R})}W_n(\tilde{R})),
\]
where the last isomorphism follows from \cite[Lemma~4.20]{Prisms}. Thus, the result holds for~$\Tilde{R}$.

Furthermore, the argument above applies to any perfectoid ring $Q$ in place of $\hat{\overline{\mathbb{Z}}}_p$ and any perfectoid ring $Q'$ under $Q$ in place of  $\Tilde{R}$. Indeed, the proof remains essentially unchanged and relies solely on the base change (in case of perfectoid rings) and \'etale localization properties for the de~Rham--Witt complex and the prismatic complex. Therefore, we conclude that if the result is true for some perfectoid $Q$, then it also holds for any perfectoid ring $Q'$ under $Q$.

Finally, we prove the comparison isomorphism for $R$, assuming the result holds for~$\Tilde{R}$. Since $R\to\Tilde{R}$ is $p$-completely faithfully flat and both $A_{\inf}(R)$ and $A_{\inf}(\Tilde{R})$ are $d$-torsion-free, it follows that $A_{\inf}(R)\to A_{\inf}(\Tilde{R})$ is $(p,d)$-completely faithfully flat. Consequently,  $W_n(R)/p^m\to W_n(\Tilde{R})/p^m$ must be faithfully flat. By Theorem~\ref{Cat} and Proposition~\ref{map}, we have $W_n(R)\simeq A_{\inf}(R)/d\phi(d)\cdots\phi^{n-1}(d)$.  Since both rings are perfectoid, by \cite[Proposition~10.14]{BMS_I},  we see that
\[
\begin{aligned}
W_n\Omega^{i, cont}_{S/R}/p^m\otimes_{W_n(R)/p^m}W_n(\Tilde{R})/p^m &\simeq (W_n\Omega^{i, cont}_{S/R}\otimes_{W_n(R)} W_n(\Tilde{R}))/p^m \\ &\simeq W_n\Omega^{i, cont}_{S_{\Tilde{R}}/\Tilde{R}}/p^m.
\end{aligned}
\]
Furthermore, since $W_n(R)/p^m\to W_n(\Tilde{R})/p^m$ is faithfully flat, by Lemma~\ref{superlemma}, we know that
\[H^i(\Prism_{S/A_{\inf}(R)}\otimes^L_{A_{\inf}(R)}W_n(R))/p^m\otimes^L_{W_n(R)/p^m}W_n(\Tilde{R})/p^m
\]
is isomorphic to
\[
H^i(\Prism_{S/A_{\inf}(R)}\otimes^L_{A_{\inf}(R)}W_n(\Tilde{R}))/p^m\simeq H^i(\Prism_{S_{\Tilde{R}}/A_{\inf}(\Tilde{R})}\otimes^L_{A_{\inf}(\Tilde{R})}W_n(\Tilde{R}))/p^m,
\]
where the last isomorphism follows from \cite[Lemma~4.20]{Prisms}.
Here $S_{\Tilde{R}}$ denotes  the base change $S\widehat{\otimes}_R \Tilde{R}$.

Next, we apply faithfully flat descent for $W_n(R)/p^m\to W_n(\Tilde{R})/p^m$. Let us denote $W_n(R)/p^m$ by $B$, $W_n(\Tilde{R})/p^m$ by $C$, $W_n\Omega^i_{S/R}/p^m$ by $M$, and the prismatic cohomology $H^i(R\Gamma_\Prism(S/A_{\inf}(R))\otimes^L_{A_{\inf}(R)}W_n(R))/p^m$ by $N$. Since $B\to C$ is faithfully flat, we have two exact Amitsur complexes:

\[
0\to M\to M\otimes_B C \to M\otimes_B C\otimes_B C\to\ldots
\]
and
\[
0\to N\to N\otimes_B C\to N\otimes_B C\otimes_B C\to \ldots.
\]

From our previous discussion, we know that $M\otimes_B C$ and $N\otimes_B C$ are isomorphic, as they appear in the comparison isomorphism for $\Tilde{R}\mod{p^m}$. Furthermore, since $\Tilde{R}\otimes_R\Tilde{R}$ is perfectoid over $\Tilde{R}$, by the first part of the proof, we know the comparison isomorphism holds for it. The same argument as above shows that the terms of that isomorphism $\mod{p^m}$ are precisely $M\otimes_B C\otimes_B C$ and $N\otimes_B C\otimes_B C$. Using the same machinery for other terms of the Amitsur complexes, we obtain the following commutative diagram:
\[
\xymatrix{
0\ar[r]& M\ar[r]\ar@{-->}[d] & M\otimes_B C \ar[r]\ar[d]^{\simeq} & M\otimes_B C\otimes_B C \ar[d]^{\simeq}\ar[r] & M\otimes_B C\otimes_B C\otimes_B C\ar[d]^{\simeq}\ar[r] &\ldots   \\
0 \ar[r] & N \ar[r] & N\otimes_B C\ar[r] & N\otimes_B C\otimes_B C \ar[r] &N\otimes_B C\otimes_B C\otimes_B C\ar[r] & \ldots,}
\]
where all solid vertical arrows are isomorphisms. Therefore, we conclude that the dashed vertical arrow gives an isomorphism between $M$ and $N$.

Thus, we establish the  comparison isomorphism $\mod{p^m}$ for $R$. After taking the limit, we obtain the result for $R$. This completes the proof for any perfect prism~$(A, d)$ such that $A/d$ is $p$-torsion-free.
\end{proof}

\begin{Rem}\label{75}
Let us reformulate the previous result explicitly for the $A/d$-algebra $S=A/d\langle T_1,\ldots, T_k \rangle$, which is the $p$-adic completion of the  polynomial algebra $S_0=A/d[T_1, \ldots, T_k]$ over $A/d$. By definition, we have $W_n\Omega^{i, cont}_{S/(A/d)}\simeq W_n\Omega^{i, cont}_{S_0/(A/d)}$. In what follows, we will frequently encounter the notation $\bigoplus\bigoplus$. By this, we mean $\bigoplus_{a\colon \{1,\ldots, k\}\to p^{-n}\mathbb{Z}_{\geq 0}}\bigoplus_{(I_0,\ldots, I_i)\in P_a}$, using the notations of Theorem~\ref{poly}. From this theorem, we know that 
\[\bigoplus\bigoplus V^{u(a)}W_{n-u(a)}(A/d)\simeq W_n\Omega^i_{S_0/(A/d)}. 
\]
Since $(A, d)$ is perfect, by Proposition~\ref{map} and Remark~\ref{rem2}, we obtain the isomorphism of $A/d\phi(d)\cdots \phi^{n-1}(d)$-modules
\[\bigoplus\bigoplus V^{u(a)}W_{n-u(a)}(A/d)\simeq \bigoplus\bigoplus A/d\phi(d)\cdots\phi^{n-1-u(a)}(d). 
\]
More precisely, we should write the latter term as
\[
\bigoplus\bigoplus \phi^{n-u(a)}(d)\cdots\phi^{n-1}(d)(A/d\phi(d)\cdots\phi^{n-1-u(a)}(d)),
\]
however, we will omit the multiplier $\phi^{n-u(a)}(d)\cdots\phi^{n-1}(d)$ in what follows, considering it as part of $A/d\phi(d)\cdots\phi^{n-1}(d)$-module structure on $A/d\phi(d)\cdots\phi^{n-1-u(a)}(d)$.

Finally, we can write $W_n\Omega^{i, cont}_{S_0/(A/d)}$ as
\[
\begin{aligned}
\varprojlim_{s} W_n\Omega^i_{(S_0/p^s)/((A/d)/p^s)}&\simeq\varprojlim_s W_n\Omega^i_{S_0/(A/d)}/p^s\\&\simeq\varprojlim_s \bigoplus\bigoplus A/d\phi(d)\cdots\phi^{n-1-u(a)}(d)/p^s,
\end{aligned}
\]
which yields the following isomorphism:
\[
\left(\bigoplus\bigoplus A/d\phi(d)\cdots\phi^{n-1-u(a)}(d)\right)^{\wedge_p}\simeq H^i(\Prism_{S/A}\otimes^L_{A}A/d\phi(d)\cdots\phi^{n-1}(d)).
\]
\end{Rem}

We now aim to extend this result to more general settings.

\begin{Lemma}\label{polynomial}
Let $(\bar{A}, d)\to (A, d)$ be a map of prisms, and let $S=A/d\langle T_1, \ldots, T_k\rangle$ be a $p$-completed polynomial $A/d$-algebra. Assume that $\phi^i(d)$ are non-zero-divisors in both $\bar{A}$ and $A$ for all $i\geq 0$. Further, suppose that
\[
\left(\bigoplus\bigoplus \bar{A}/d\phi(d)\cdots\phi^{n-1-u(a)}(d)\right)^{\wedge_p}\simeq H^i(\Prism_{\bar{S}/\bar{A}}\otimes^L_{\bar{A}}\bar{A}/d\phi(d)\cdots\phi^{n-1}(d)),
\]
where $\bar{S}=\bar{A}/d\langle T_1,\ldots, T_k\rangle$. Then there exists an explicit functorial isomorphism 
\[
\left(\bigoplus\bigoplus A/d\phi(d)\cdots\phi^{n-1-u(a)}(d)\right)^{\wedge_p}\simeq H^i(\Prism_{S/A}\otimes^L_{A}A/d\phi(d)\cdots\phi^{n-1}(d))
\]
of $A/d\phi(d)\cdots\phi^{n-1}(d)$-modules.
\end{Lemma}

\begin{proof}
From our assumptions, we know that 
\[
\left(\bigoplus\bigoplus \bar{A}/d\phi(d)\cdots\phi^{n-1-u(a)}(d)\right)^{\wedge_p}\simeq H^i(\Prism_{{\bar{S}/\bar{A}}}\otimes^L_{\bar{A}}\bar{A}/d\phi(d)\cdots\phi^{n-1}(d)),
\]
where $\bar{S}=\bar{A}/d\langle T_1,\ldots, T_k\rangle$. From the non-zero-divisor condition, we have 
\[
\left(\bigoplus\bigoplus \bar{A}/d\cdots\phi^{n-1-u(a)}(d)\right)^{\wedge_p}\widehat{\otimes}^L_{\bar{A}}A\simeq \left(\bigoplus\bigoplus A/d\cdots\phi^{n-1-u(a)}(d)\right)^{\wedge_p}
\]
and
\[
\left(\bigoplus\bigoplus \bar{A}/d\cdots\phi^{n-1-u(a)}(d)\right)^{\wedge_p}\widehat{\otimes}_{\bar{A}}A\simeq \left(\bigoplus\bigoplus A/d\cdots\phi^{n-1-u(a)}(d)\right)^{\wedge_p}.
\]
Let us denote $\Prism_{{\bar{S}/\bar{A}}}\otimes^L_{\bar{A}}\bar{A}/d\phi(d)\cdots\phi^{n-1}(d)$ by $T$.
By the previous observation and Remark~\ref{Completed}, we see that 
\[H^i(T\widehat{\otimes}^L_{\bar{A}}A)\simeq H^i(T)\widehat{\otimes}_{\bar{A}}A.
\]
This precisely means that 
\[ H^i(\Prism_{S/A}\otimes^L_{A}A/d\phi(d)\cdots\phi^{n-1}(d))\simeq \left(\bigoplus\bigoplus A/d\phi(d)\cdots\phi^{n-1-u(a)}(d)\right)^{\wedge_p},
\]
which gives us the desired result.
\end{proof}

Lemma~\ref{polynomial} shows that under the given non-zero-divisor condition, the explicit isomorphism for $\bar{A}$ implies the corresponding result for $A$.

\begin{Prop}\label{lalalala}
Let $(A, d)$ be either the universal oriented prism or a free prism over the universal one, and let $S=A/d\langle T_1,\ldots, T_k\rangle$. Then there exists an explicit functorial isomorphism
\[\left(\bigoplus\bigoplus A/d\phi(d)\cdots\phi^{n-1-u(a)}(d)\right)^{\wedge_p}\simeq H^i(\Prism_{S/A}\otimes^L_{A}A/d\phi(d)\cdots\phi^{n-1}(d))
\]
of $A/d\phi(d)\cdots\phi^{n-1}(d)$-modules.
\end{Prop}

\begin{proof}
By Remark~\ref{75}, we already have an explicit isomorphism for $A_\infty$. Moreover, by \cite{Prisms}, we know that $A\to A_\infty$ is $(p,d)$-completely faithfully flat. We now apply the faithfully flat descent argument, similar to that used in the proof of Theorem~\ref{Proof}. For this purpose, recall that, by \cite[Proposition~1.9]{Tian}, the Amitsur complex $A\to (A_\infty)^{\otimes\bullet}$, associated to the \v{C}ech nerve $(A_\infty)^{\otimes\bullet}$ of $A\to A_\infty$ is exact. Therefore, it suffices to establish the isomorphisms for all terms of this nerve. By the explicit construction of the universal prism, each term $(A_\infty)^{\otimes k}$ satisfies the non-zero-divisor condition specified in Lemma~\ref{polynomial}. Applying this lemma, we inductively obtain the required isomorphisms for all terms of the \v{C}ech nerve of $A\to A_\infty$, allowing us to apply descent along it to get the explicit isomorphism for $A$.
\end{proof}

We can now extend this result to a wider class of prisms.

\begin{Prop}\label{res1}
Let $(A, d)$ be a prism, and let $S=A/d\langle T_1,\ldots,T_k\rangle$ be a $p$\nobreakdash-complet\-ed polynomial $A/d$-algebra. If  $\phi^i(d)$ are non-zero-divisors in $A$ for all $i\geq 0$, then there exists a functorial isomorphism
\[\left(\bigoplus\bigoplus A/d\phi(d)\cdots\phi^{n-1-u(a)}(d)\right)^{\wedge_p}\simeq H^i(\Prism_{S/A}\otimes^L_{A}A/d\phi(d)\cdots\phi^{n-1}(d))
\]
of $A/d\phi(d)\cdots\phi^{n-1}(d)$-modules.
\end{Prop}

\begin{proof}
We choose a surjection $\bar{A}\to A$ of prisms, where  $\bar{A}$ is a free prism over the universal one, and then apply Lemma~\ref{polynomial} and Proposition~\ref{lalalala} to obtain the result.
\end{proof}

\appendix

\section{Alternative proof of Proposition~\ref{map}}\label{appendix}
In this appendix, we present an alternative method for constructing maps 
\[
W_n(A/d)\to A/d\phi(d)\cdots\phi^{n-1}(d),
\]
proposed by an anonymous referee. Unlike our original proof, this approach does not provide an explicit formula for the images of the generators in  $\prod_{i=0}^{n-1}A/\phi^i(d)$, but it is somewhat simpler.

\begin{Prop}\label{map1}
For any prism $(A, d)$, there exists a functorial map 
\[W_n(A/d)\to A/d\phi(d)\cdots\phi^{n-1}(d).
\]
\end{Prop}
\begin{proof}
First, assume that $(A,d)$ is a prism such that $A/d$ is $p$-torsion-free. In this case, by faithfully flat descent, it suffices to work with $(A, d)$, where the Frobenius $\phi\colon A\to A$ is surjective. Indeed, the proof is similar to the argument using Amitsur complexes in Proposition~\ref{map}, one just needs to check that the Frobenius on $A_\infty\widehat{\otimes}_AA_\infty$ is surjective, which is straightforward. Therefore, we may assume that the Frobenius is surjective on $A$.

Under this assumption, by the universal property of Witt vectors, we obtain a map of $\delta$-rings $A\to W(A/d)$ such that the composite $A\to W(A/d)\to W_n(A/d)$ is surjective. Next, consider the composition of $\phi^{n-1}\colon A\to A$ with the natural projection $A\to A/d\phi(d)\cdots\phi^{n-1}(d)$. We claim that this composition factors through the surjection $A\twoheadrightarrow W_n(A/d)$. For this, it is enough to show that if $x\in A$ maps to zero in $W_n(A/d)$, then $\phi^i(x)\equiv 0 \mod{d}$ for all $0\leq i\leq n-1$. Indeed, if these congruences hold, then, by applying the corresponding power of $\phi$, it follows that $y=\phi^{n-1}(x)\equiv\ 0 \mod{\phi^i(d)}$ for $0\leq i\leq n-1$. The first argument of Proposition~\ref{map} provides an inclusion  $A/d\phi(d)\cdots\phi^{n-1}(d)\hookrightarrow\prod_{i=0}^{n-1}A/\phi^i(d)$ (also see Remark~\ref{Bras}). This inclusion implies that $y=\phi^{n-1}(x)\equiv 0 \mod{d\phi(d)\cdots\phi^{n-1}(d)}$, which is our desired result.

Thus, we need to show that if $x\in A$ maps to zero in $W_n(A/d)$, then $\phi^i(x)\in (d)$. For this, consider the Witt vector Frobenius $F^i\colon W_n(A/d)\to W_{n-i}(A/d)$ composed with the projection $W_{n-i}(A/d)\to A/d$. From the commutative diagram
\[\xymatrix{
A\ar[r]\ar[d]^{\phi^i}  & W_n(A/d)\ar[d]^{F^i} \\
A\ar[r] & W_{n-i}(A/d),}
\]
we see easily that $\phi^i(x)\in (d)$ as desired.

The extension to an arbitrary oriented prism $(A,d)$ follows the same argument as the final paragraph of the proof of Proposition~\ref{map}, so we omit it here.
\end{proof}

\begin{Rem}
    From the argument above, it is clear that the constructed map
    \[
    W_n(A/d)\to A/d\phi(d)\cdots\phi^{n-1}(d)
    \]
    is the unique map for which the composition  $A\to W_n(A/d)\to A/d\phi(d)\cdots\phi^{n-1}(d)$ coincides with $A\xrightarrow{\phi^{n-1}} A\to A/d\phi(d)\cdots\phi^{n-1}(d)$.
\end{Rem}

\begin{Rem}
We also note that a more conceptual construction of this map was later obtained by Yuri Sulyma in \cite{Sul}. His approach leverages the relationship between prisms and Tambara functors. 
\end{Rem}


\begin{bibdiv}
\begin{biblist}

\bib{An}{article}{

title={La conjecture du facteur direct},
author={Yves Andr\'e},
journal={Publ. Math. Inst. Hautes \'Etudes Sci.},
volume={127},
date={2018},
pages={71--93}
}

\bib{Arr}{article}{

title={Prismatic Dieudonn\'e theory},
author = {Johannes Ansch\"utz},
author={Arthur-C\'esar Le Bras},
journal={arXiv:1907.10525v3},
date={2021},
} 


\bib{Ar}{article}{

title={The $p$-completed cyclotomic trace in degree 2},
author = {Johannes Ansch\"utz},
author={Arthur-C\'esar Le Bras},
journal={Annals of K-Theory},
volume={5},
number = {3},
date={2020},
pages={539--580}
}


\bib{Bhatt}{article}{

title={Geometric aspects of $p$-adic Hodge theory: prismatic cohomology},
journal ={Eilenberg Lectures at Columbia University},
author={Bhargav Bhatt},
note={Available at \url{https://www.math.ias.edu/~bhatt/teaching/prismatic-columbia/}}
}

\bib{BHLURIE}{article}{

title={Absolute prismatic cohomology},
author = {Bhargav Bhatt},
author={Jacob Lurie},
journal={arXiv:2201.06120},
date={2022},
} 



\bib{BMS_I}{article}{

title={Integral $p$-adic Hodge theory},
author={Bhargav Bhatt},
author={Mathew Morrow},
author={Peter Scholze},
journal={Publ. Math. Inst. Hautes \'Etudes Sci.},
volume={128},
date={2018},
pages={219--397}
}



\bib{Prisms}{article}{

title={Prisms and prismatic cohomology},
author={Bhargav Bhatt},
author={Peter Scholze},
journal={Ann. of Math.},
volume={196},
number={3},
date={2022},
pages={1135--1275}
}


\bib{Borg}{article}{

title={Witt vectors, lambda-rings, and arithmetic jet spaces},
journal ={Course at the University of Copenhagen},
author={James Borger},
note={Available at \url{https://maths-people.anu.edu.au/~borger/classes/copenhagen-2016/index.html}}
}


\bib{Il}{article}{

title={ Complexe de de Rham-Witt et cohomologie cristalline},
author={Luc Illusie},
journal={Ann. Sci. Ecole Norm. Sup.},
volume={12},
date={1979},
number = {4},
pages={501--661}
}


\bib{Ill}{article}{

title={De Rham-Witt complexes and p-adic Hodge theory},
author={Luc Illusie},
journal={Pre-notes for Sapporo seminar},
note={Available at \url{https://www.math.u-psud.fr/~illusie/Illusie-Sapporo1a.pdf}},
date={2011}
}

\bib{IlR}{article}{

title={Les suites spectrales associees au complexe de de Rham-Witt},
author={Luc Illusie},
author={Michel Raynaud},
journal={Publ. Math. Inst. Hautes \'Etudes Sci.},
volume={57},
date={1983},
pages={73--212}
}


 
\bib{Joy}{article}{

title={ $\delta$-anneaux et vecteurs de Witt},
author={Andr\'e Joyal},
journal={C. R. Math. Rep. Acad. Sci. Canada},
volume={7},
date={1985},
number = {2},
pages={177--182}
}

\bib{LZ}{article}{

title={De Rham-Witt cohomology for a proper and smooth morphism},
author={Andreas Langer},
author={Thomas Zink},
journal={J. Inst. Math. Jussieu},
volume={3},
number={2},
date={2004},
pages={231--314},
}

\bibitem{stacks-project}
The Stacks Project Authors, \textit{Stacks Project},
\url{http://stacks.math.columbia.edu}, retrieved January 2025.


\bib{Sul}{article}{

title={Prisms and Tambara functors I: Twisted powers, transversality, and the perfect sandwich},
author = {Yuri~J.~F.~Sulyma},
journal={arXiv:2309.03181},
date={2023},
} 




\bib{Tian}{article}{

title={Finiteness and Duality for the cohomology of prismatic crystals},
author={Yichao Tian},
journal={Journal f\"ur die reine und angewandte Mathematik (Crelles Journal)},
date={2023},
volume={2023},
number={800},
pages={217--257}
}






\end{biblist}

\end{bibdiv}
\end{document}